\documentclass[a4paper,11pt]{article}
\usepackage{amsmath,amsthm,amssymb,amsfonts,graphicx, bm}
\usepackage[top=3cm, bottom=3cm, left=3cm, right=3cm]{geometry}
\usepackage[shortlabels]{enumitem}
\usepackage{booktabs}

\usepackage{hyperref}

\usepackage{tikz}
\usetikzlibrary{decorations.pathmorphing}
\usetikzlibrary{arrows.meta,arrows}
\usetikzlibrary{shadings,calc}

\newtheorem{question}{Question}[section]

\makeatletter
\newcommand{\dotminus}{\mathbin{\text{\@dotminus}}}

\newcommand{\@dotminus}{%
  \ooalign{\hidewidth\raise.15ex\hbox{$\scriptstyle\circ$}\hidewidth\cr$\m@th-$\cr}%
}
\makeatother

\makeatletter
\newcommand{\dotpplus}{\mathbin{\text{\@dotpplus}}}

\newcommand{\@dotpplus}{%
  \ooalign{\hidewidth\raise.15ex\hbox{$\scriptstyle\circ$}\hidewidth\cr$\m@th+$\cr}%
}
\makeatother

\newtheorem{observation}{Observation}[section]

\newcommand{\R}{\mathbb{R}}

\newcommand{\K}{\ensuremath{\mathcal{K}}}

\newcommand{\F}{\ensuremath{\mathcal{F}}}
\newcommand{\sym}{\ensuremath{\mathcal{S}}}

\newcommand{\cp}{\ensuremath{\mathcal{CP}}}
\newcommand{\cop}{\ensuremath{\mathcal{COP}}}
\newcommand{\cc}{\textnormal{cc\,}}

\newcommand{\cpr}{\textnormal{cpr\,}}
\newcommand{\plmod}{\ensuremath{\dotpplus}}

\newcommand{\vc}[1]{\ensuremath{{\bf #1}}}
\newcommand{\vct}[1]{\ensuremath{{\bf #1}^T\!}}
\newcommand{\vi}[2]{\ensuremath{{\bf {#1}}_{#2}}}
\newcommand{\vit}[2]{\ensuremath{{\bf {#1}}_{#2}^T\!}}

\newcommand{\vdn}[3]{\ensuremath{\vi {#1}{#2},\ldots,\vi {#1}{#3}}}
\newcommand{\vvt}[1]{\ensuremath{\vc {#1} \vct {#1}}}
\newcommand{\vvti}[2]{\ensuremath{{\bf #1}_{#2}^{\phantom{t}} \vit {#1}{#2}}}

\newtheorem{defin}{Definition}[section]

\newtheorem{theorem}[defin]{Theorem}
\newtheorem{remark}[defin]{Remark}
\newtheorem{corollary}[defin]{Corollary}
\newtheorem{lemma}[defin]{Lemma}
\newtheorem{example}[defin]{Example}

\DeclareMathOperator{\trace}{trace}
\DeclareMathOperator{\rank}{rank}

\DeclareMathOperator{\cone}{cone}

\DeclareMathOperator{\supp}{supp}

\newcommand{\bR} {\ensuremath{\mathbb{R}}}

\author{Naomi Shaked-Monderer\thanks{The Max Stern Yezreel Valley College, Yezreel Valley 1930600, Israel.
Email: nomi@technion.ac.il }}
\title{On the number of CP factorizations\\ of a completely positive matrix}

\date{\today}

\begin{document}
\maketitle
\begin{abstract}
\noindent A square matrix $A$ is completely positive if $A=BB^T$, where $B$ is a (not necessarily square)
nonnegative matrix. In general, a completely positive matrix may have many, even infinitely many,
such CP factorizations. But in some cases a unique CP factorization exists. We prove a simple necessary and sufficient
condition for a completely positive matrix whose graph is triangle free to
have a unique CP factorization. This implies uniqueness
of the CP factorization for some other matrices  on the boundary of the cone $\cp_n$ of $n\times n$
completely positive matrices. We also describe the minimal face of $\cp_n$ containing a completely positive $A$.
If $A$ has a unique CP factorization, this face is polyhedral.

\medskip



\end{abstract}

\section{introduction}
An $n\times n$ matrix $A$ is said to be {\it completely positive} if $A=BB^T$
for some (entrywise) nonnegative $B$.
The matrix $B$ need not be square, and typically
the completely positive $A$ has many, often infinitely many, such factorizations.
The minimal number of columns in such a nonnegative factor $B$ is called the {\it cp-rank}
of $A$, and denoted by $\cpr A$. A  factorization $A=BB^T$, $B\ge 0$, is called a {\it CP factorization} of $A$,
and if $B$ has $\cpr A$ columns, a {\it minimal CP factorization}. These are also, in general, not unique.
In this paper we consider
the number of CP factorizations and minimal CP factorizations of a given
completely positive matrix.

Determining whether a given matrix is completely positive
is hard (in complexity terms, NP-hard --- see \cite{DickinsonGijben2014}).
Finding the cp-rank of a known completely positive matrix is also hard: one CP
factorization does not lead naturally to all others, or to a minimal one. These difficult problems
have gained an increased interest in the last couple of decades, due to
the use of completely positive matrices in optimization. The $n\times n$ completely
positive matrices form a convex cone $\cp_n$ in the space $\sym_n$ of $n\times n$
symmetric matrices. The dual of this cone is the cone $\cop_n$ of $n\times n$ copositive matrices.
(An $n\times n$ matrix $A$ is {\it copositive} if $\vct xA\vc x\ge 0$ for every
nonnegative vector $\vc x\in \R^n$.) It has been shown that a wide family of difficult optimization problems,
both discrete and continuous, may be reformulated as linear optimization problems with
a matrix variable  in either $\cp_n$ or $\cop_n$. In such a formulation, the difficulty
of the problem is shifted to the cone. Hence the increased interest in
the matrices in the cones $\cp_n$ and $\cop_n$, and also
in the structure of these cones.
For details on copositive optimization see \cite{Duer2010} and \cite{Bomze2012}.
 See \cite{BermanShaked2003} for a survey on complete positivity. An updated
and extended book  by the same authors, titled {\it Copositive and Completely Positive Matrices}, is
due to appear soon, containing, as the name suggests, also an up-to-date survey on copositivity.
See \cite{BermanDuerShaked2015} for open problems in this field. Although some work has been
done on these problems since the paper appeared, they are still generally open.

The number of CP factorizations has not
been widely explored so far, but some results do exist. The number of minimal CP factorizations
was considered in \cite{ZhangLi2000}
for completely positive matrices
whose graph is a cycle, in \cite{DickinsonDuer2012} for matrices that have an acyclic or a unicyclic graph,
and in \cite{ShakedBomzeJarreSchachinger2013} for certain positive matrices on the boundary of $\cp_5$.
In \cite{Shaked2013}
completely positive matrices with a chordal graph and minimal possible rank were shown to have
 a unique CP factorization.
Even fewer papers deal with the structure of $\cp_n$. Some geometry results may be found in \cite{Dickinson2011},
and faces of $\cp_5$ were studied in \cite{Zhang2018,Zhang2019}. (The latter two papers contain also  examples
of  matrices in $\cp_5$ that have unique CP factorizations.)

The main result in this paper is a simple necessary and sufficient
condition for a completely positive matrix whose graph is triangle free to
have a unique CP factorization. This result generalizes both the results of \cite{ZhangLi2000}
and of \cite{DickinsonDuer2012} in two ways:
we consider the wider family of matrices with triangle free graphs, and consider also CP factorizations that are not minimal.
We will also show that this result may be applied to find the number of CP factorizations
of additional matrices on the boundary of $\cp_n$  (some of them positive), in particular those
considered in \cite{ShakedBomzeJarreSchachinger2013}.  We conclude with
a description of the minimal face of $\cp_n$ containing a completely
positive $A$ --- both in the general case  and in the special case that $A$ has a unique CP representation, demonstrating
the implication of this uniqueness.

The paper is organized as follows: In the next section we establish mainly notation and terminology. In Section \ref{sec:initial}
we make some initial observations
and state some results regarding the number of CP factorizations of a completely positive matrix.
In Section \ref{sec:main} we state and prove
the main result about matrices with triangle free graphs, and demonstrate
its relevance to other matrices on the boundary.
In Section \ref{sec:faces} we describe the minimal face of $\cp_n$ containing a completely
positive $A$ in the general case, and in the case that $A$ has a unique CP representation.

\section{Notation and terminology}\label{sec:notback}
Vectors are denoted by lower bold case letters, with the $i$-th entry of a vector $\vc x$ denoted by $x_i$.  Matrices are denoted
by capital letters.  In particular, $\vdn e1n$ are
the standard basis vectors in $\R^n$, $\vc 0$ is the zero vector, and $\vc e$ is the vector of
all ones. The $n\times n$ zero matrix  is $0_n$,  and   $E_{ij}$ is a square matrix all of whose entries are zero, except for 1 in the $i,j$ position.
For a vector $\vc x\in \R^n$,  $\supp \vc x=\{1\le i\le n\mid x_i\ne 0\}$ is the {\it support} of $\vc x$.
  We say that $\vc x$ is a {\it nonnegative} ({\it positive}) vector, and write $\vc x\ge \vc 0$ ($\vc x>\vc 0$), if all
 its entries are nonnegative (positive); $\R^n_+=\{\vc x\mid \vc x\ge \vc 0\}$ is the nonnegative orthant of $\R^n$.
 Similarly, we write $A\ge 0$ ($A>0$) and say that $A$ is {\it nonnegative} ({\it positive}) if all the entries of the matrix $A$ are nonnegative (positive).
 A diagonal matrix is called {\it positive} if all its diagonal elements are
 positive. A {\it signature matrix} is a diagonal matrix with 1 and -1 entries on the diagonal.
 If $\sigma\subseteq \{1, \dots, n\}$, we denote by $\vc x[\sigma]$ the restriction of the
 vector $\vc x\in \R^n$ to the indices in $\sigma$, and by $A[\sigma]$ the principal submatrix of an $n\times n$
 matrix $A$ whose rows and columns are indexed by $\sigma$. We denote by $A(\sigma)$ the principal
 submatrix of $A$ on  the rows and columns indexed by the complement of $\sigma$, $A[\sigma|\sigma)$ is
 the submatrix whose rows are indexed by $\sigma$, and its columns by the complement of $\sigma$, and $A(\sigma|\sigma]$
 is submatrix whose rows are indexed by the complement of $\sigma$, and its rows by the complement of $\sigma$.
 If $A[\sigma]$ is nonsingular, the {\it Schur complement} of $A[\sigma]$ in $A$ is defined as
 \[A/A[\sigma]=A(\sigma)-A(\sigma|\sigma]A[\sigma]^{-1}A[\sigma|\sigma).\]
An {\it $M$-matrix} is a square matrix of the form $dI-Q$, where $Q$ is a nonnegative matrix, and $d\ge \rho(Q)$,
the spectral radius of $Q$. The {\it comparison matrix}
of a square matrix $A$, $M(A)$, is defined by
\[M(A)_{ij}=\left\{\begin{array}{rl} |a_{ij}|, \quad & \mbox{if }\, i=j \\
                                     -|a_{ij}|, \quad & \mbox{if }\, i\ne j\end{array}\right. .\]
Some details on these matrix notions may be found in \cite[Chapter 1]{BermanShaked2003}.
A good general reference on matrix theory is \cite{HornJohnson2013}, and \cite{BermanPlemmons1994}
is a recommended reference for $M$-matrices, their many equivalent definitions and their properties. For our
purposes it is important that  a symmetric matrix with nonnegative diagonal elements and nonpositive
off-diagonal elements is an $M$-matrix if and only if it is positive semidefinite, and that an irreducible symmetric
$M$-matrix has a positive eigenvector corresponding to its smallest eigenvalue, which is simple.
This implies that a symmetric $M$-matrix $S$ is positive semidefinite if and only if for some positive diagonal $D$
the matrix $DSD$ is diagonally dominant. Note also that $M(A)$ is diagonally dominant if and only if $A$ is
diagonally dominant, and recall that  an irreducible diagonally dominant matrix is nonsingular if and only if
in at least one row the diagonal dominance is strict.
Regarding the Schur
complement one should keep in mind that if $A$ is positive semidefinite and $A[\sigma]$ is nonsingular,
then $A/A[\sigma]$ is a positive semidefinite matrix and $\rank A/A[\sigma]=\rank A-\rank A[\sigma]$.

We use basic graph theoretic notions, which may be found in standard textbooks on graph theory, see, e.g. \cite{Diestel2018}.
We only consider graphs that are simple (undirected, no loops, no multiple edges).
We denote by $\cc(G)$ the {\it (edge-)clique covering number} of a graph $G$, the minimal number of cliques needed to
cover $G$'s edges. A graph is {\it triangle free} if the largest cliques in the graph are its edges, hence for
a triangle free graph $\cc(G)$ is the number of edges in $G$.
The {\it graph} of a symmetric $n\times n$ matrix $A$, denoted by $G(A)$, has vertex set
$\{1, \dots, n\}$ and $\{i,j\}$ is an edge if and only if $a_{ij}\ne 0$. The matrix $A$ is irreducible if and only
if $G(A)$ is connected.

We will recall known results on completely positive matrices and the cones $\cp_n$ and $\cop_n$
as we go along. When no explicit reference is given, references can be found in \cite{BermanShaked2003}.
For now, we only mention this:
An obvious necessary condition for $A$ to be completely positive is that it is
{\it doubly nonnegative}, i.e., both positive semidefinite and entrywise nonnegative. However, this necessary condition
is not sufficient. A graph $G$ is called {\it completely positive}
if every doubly nonnegative matrix with this graph is completely positive. Completely positive
graphs were fully characterized: these are exactly the graphs that contain no odd cycle
on $5$ vertices or more. In particular, all bipartite graphs are completely positive.

\section{Initial results on the number of CP factorizations}\label{sec:initial}
We first point out that the CP factorization $A=BB^T$, $B\ge 0$, is equivalent to
\[A=\sum_{i=1}^k \vvti bi, \quad \vi bi\ge \vc 0, ~i=1, \dots, k,\]
where $\vdn b1k$ are the columns of $B$. We refer to this sum as a {\it CP representation}, and if the CP factorization is minimal,
as a {\it minimal CP representation}. We use CP factorizations and CP representations
interchangeably.
To avoid artificially extended CP factorizations, we only consider CP factorizations
$A=BB^T$  where in the nonnegative $B$ no column is a scalar multiple of another (that is,
the columns of $B$ are {\it pairwise linearly independent}). Two
CP factorizations $A=BB^T$ and $A=CC^T$ are considered equal if
$C$ and $B$ only differ by the order of their columns (that is,
$C=BP$ for some permutation matrix $P$).

If  $A=A_1\oplus A_2\oplus \cdots \oplus A_m$,
then the number of CP factorizations of $A$ is completely determined by
the number of CP factorizations of the direct summands.
The matrix $A$
has infinitely many CP factorizations if and only if at least for one $i$
the matrix $A_i$ has infinitely many CP factorizations. If $A_i$
has $k_i$ different CP factorizations, then $A$ has $\Pi_{i=1}^mk_i$ different
CP factorizations. We may therefore simplify the discussion by restricting out attention to irreducible
matrices.

As usual when considering completely positive matrices, we rely on
the fact that if $A$ is a symmetric matrix, $P$ is a permutation matrix
and $D$ is a positive diagonal matrix, all of the same order, then
$A$ is completely positive matrix if and only if $PAP^T$ is, if
and only if $DAD$ is. For our purposes it is also important to note that
\[A=BB^T \quad\Leftrightarrow \quad DAD=(DB)(DB)^T,\]
\[A=BB^T \quad\Leftrightarrow \quad PAP^T=(PB)(PB)^T.\]
Since such $P$ and $D$ have nonnegative inverses, these equivalences imply a one-to-one correspondence
between CP factorizations of $A$  and those of $DAD$, or $PAP^T$.
We  may therefore apply diagonal scaling and permutation similarity
to our matrices without changing the number of (minimal) CP factorizations.

With these basic observations in mind, we state some initial results
regarding the number of CP factorizations.

An obvious example of a
completely positive matrix with a unique CP factorization is any rank 1
matrix $\vvt b$, $\vc b\in \R^n_+$. For matrices of rank 2, we recall an idea that dates back to \cite{Hall1998},
and reestablished since then in many papers, including \cite{DickinsonDuer2012}, in several variations.
Since this is a basis for several
more observations, we include a proof. Recall that the cp-rank satisfies $\cpr A\ge \rank A$, and may
be much larger than the rank.
However, equality holds if $\rank A=1$ or $\rank A=2$.

\begin{observation}\label{ob:b1b2}
Let $\vi b1, \vi b2\in \R^n_+$ be linearly independent, and have $\supp \vi b1\supseteq \supp \vi b2$.
Then  $A=\vvti b1+\vvti b2$ has infinitely many minimal CP factorizations.
\end{observation}

\begin{proof}
Let $B=\left[\begin{array}{cc}
               \vi b1 & \vi b2
             \end{array}
 \right]$. Each of the $n$ rows of $B$ is in the nonnegative quadrant of $\R^2$, and has a
 positive first entry. Hence no two rows are orthogonal to each other, and
 there are infinitely many possible rotations  that keep these $n$
 vectors $\R^2_+$, hence infinitely many orthogonal $2\times 2$ matrices $R$
 such that $BR\ge 0$, with different  $R$'s  resulting in different $BR$'s.
\end{proof}

Note that if $A=BB^T$, $B\ge 0$, is a CP factorization, then the support of each column of $B$
is a clique in $G(A)$, and these cliques cover all the edges of $G(A)$. Hence
$\cpr A\ge \cc(G)$.
In \cite[Theorem 1 \& Corollary 1]{Shaked2013} it was shown that if $A$ is a completely positive
matrix with a chordal graph (a graph that has no induced cycle greater than a triangle),
and $\rank A=\cc(G(A))$, then $A$ has a unique CP representation. Combining
this result with Observation \ref{ob:b1b2}, we get the following corollary.

\begin{corollary}\label{cor:rnk 2}
An irreducible  rank $2$ completely positive matrix $A$ has a unique  CP representation if and only if
$A$ has a zero off-diagonal entry in a row and column that are nonzero. If $A$
does not have such a zero off-diagonal entry, then it has infinitely many minimal CP
representations.
\end{corollary}

\begin{proof}
Since $\cpr A=\rank A=2$, we have that $\cc(G(A))\le 2$. The   matrix $A$
has a zero off-diagonal entry in a row and column that are nonzero if and only if
$\cc(G(A))\ge 2$ (since $A$ is irreducible, each row/column has at least one positive
off-diagonal entry). Hence we need to show that if $\cc(G(A))=2$, $A$ has a
unique CP representation, and if $\cc(G(A))=1$, $A$ has infinitely many CP
representations. The first claim follows from \cite[Theorem 1 \& Corollary 1]{Shaked2013}.

If $\cc(G(A))=1$, in any minimal CP representation $A=\vvti b1+\vvti b2$
the support of one of the vectors  includes the support of the other. (Otherwise
there exist  $i\in \supp \vi b2\setminus \supp \vi b1$,
and   $j\in \supp \vi b1\setminus \supp \vi b2$, and then $a_{ij}=0$
and $\cc(G(A))\ge 2$.)
By Observation \ref{ob:b1b2} such a CP representation yields infinitely many others.
\end{proof}

The interior of $\cp_n$ was described in \cite{DuerStill2008}, and the description was refined in
\cite{Dickinson2010}. The interior consists of nonsingular matrices $A$ that have a CP factorization $A=BB^T$ in which
at least one column of $B$ is positive. From Observation \ref{ob:b1b2} we therefore get the following.

\begin{corollary}\label{cor:int}
Any matrix in the interior of $\cp_n$, $n\ge 2$, has infinitely many CP factorizations.
\end{corollary}

\begin{proof}
Let $A=BB^T$ be a CP factorization of $A$ with a positive column. Without loss of generality, $\vi b1>\vc 0$. By Observation \ref{ob:b1b2},
we may replace columns $\vi b1$ and $\vi b2$ by infinitely many other pairs of columns, to obtain infinitely many CP factorizations of $A$.
\end{proof}

However, in \cite{BomzeDickinsonStill2015} it was shown
that there are matrices in the interior of $\cp_n$ that do not have a {\it minimal} CP factorization with a positive column.
We do not know the answer to the following question:

\begin{question}\label{qst:inter}
Does there exist a matrix $A$ in the interior of $\cp_n$ which has
a unique minimal CP factorization, or finitely many such factorizations?
\end{question}

The main result of this paper generalizes  the result of \cite{ZhangLi2000}, where it was shown
that a completely positive matrix $A$  whose graph is a cycle has a unique minimal CP representation if
$M(A)$ is singular, and two if $M(A)$ is nonsingular. The proof there is by determinant computations.
The result also generalizes the results
of \cite{DickinsonDuer2012}, where the possible values of the number of minimal CP representations
were found in the case that the graph of the matrix is either a tree or unicylic.
The proofs there are by an algorithm that computes
a minimal CP factorization, and  reveals the number of possible outcomes.
Since the proof of the main theorem relies on the result for trees, we provide
here for completeness a direct proof. It uses the same principals as the ones
used in the algorithm of \cite{DickinsonDuer2012}, and in the proof for general chordal graphs
in \cite{Shaked2013}, in particular,  the known fact that an $n\times n$ completely positive
matrix whose graph is a tree has rank at least $n-1$ and cp-rank equal to the rank.

\begin{theorem}\cite{DickinsonDuer2012}\label{th:trees}
Let $A$ be an $n\times n$ completely positive matrix whose graph is a tree, $n\ge 2$.
If $A$ is singular, then $A$ has a unique minimal CP factorization.
If $A$ is nonsingular, then $A$ has infinitely many CP factorizations.
\end{theorem}

\begin{proof}
We prove the result for singular matrices by induction on $n$. For $n=2$, this is simply the
fact that a rank 1 completely positive matrix has a unique CP representation.
Suppose the result holds for $n-1\ge 2$, and let $A$ by an $n\times n$ singular completely positive matrix whose graph is a tree.
Then $\cpr A=\rank A=n-1$, and each minimal CP representation has $n-1$
columns, each supported by an edge of $G(A)$. Without loss of generality
we may assume that $1$ is a pendant vertex in $G(A)$, adjacent only to $2$.
Then
\arraycolsep=3pt
\[\arraycolsep=3pt A=\left[\begin{array}{cc} a_{11}&\!\!\begin{array}{cccc} a_{12}&0&\dots&0\end{array}\\
 \begin{array}{c} a_{12} \\0 \\ \vdots  \\~0\end{array} & A(1) \end{array} \right].\]
In any minimal CP representation of $A$, exactly
one of the vectors, say $\vi b1$, is supported by $\{1, 2\}$. Necessarily
\[\vvti b1=\left[\begin{array}{cc}
           a_{11} & a_{12} \\
           a_{12} & \frac{a_{12}^2}{a_{11}} \\
         \end{array}\right]\oplus 0_{n-2}.\]
Hence
\[A-\vvti b1=0_1\oplus A/A[1]\]
is doubly nonnegative, and its rank is one less than $\rank A$, i.e., $n-2$. The graph
of $A/A[1] $ is the tree obtained from $G(A)$ by deleting the vertex 1 and
the edge incident with it.  As trees are completely positive graphs, $A/A[1]$ is a completely positive
singular matrix, and by the induction hypothesis this matrix has a unique minimal CP represntation as
a sum of $n-2$ rank $1$ matrices. Hence $A-\vvti b1$ has a unique minimal CP representation
\[A-\vvti b1=\sum_{i=2}^n\vvti bi,\]
and $A=\sum_{i=1}^n\vvti bi$ is the only minimal CP representation of $A$.

If $A$ is an $n\times n$ nonsingular matrix whose graph is a tree, then
$\cpr A=n$. It has a minimal CP representation
\[A=\sum_{i=1}^n\vvti bi.\]
In this CP representation, $n-1$ of the vectors $\vdn b1n$  have to be supported by the $n-1$ different edges of
$G(A)$. The support of the $n$-th vector is contained in the support of one of these  $n-1$
vectors. Without loss of generality, we may assume $\supp \vi b2\subseteq \supp \vi b1$.
By Observation \ref{ob:b1b2}, one may obtain infinitely many different CP representation
by replacing $\vvti b1+\vvti b2$.
\end{proof}

Our main result generalizes also the
 results of \cite{ZhangLi2000, DickinsonDuer2012} regarding the number of CP
factorizations of completely matrices whose graph is a cycle or unicylic, except that it
does not give an indication on the maximum number of minimal CP factorizations
of a unicyclic graph. For completeness, we state
and prove here this part. The proof is different from the
original proofs, but like them relies
on the fact that a completely positive matrix whose graph is triangle free and not a tree
has cp-rank equal to the number of edges of the graph.

\begin{theorem}\cite{ZhangLi2000, DickinsonDuer2012}\label{thm:cycle}
Let $A$ be an $n\times n$ irreducible completely positive matrix whose graph contains exactly one cycle, on $k\ge 4$ vertices.
Then $A$ has at most two minimal CP factorizations.
\end{theorem}

\begin{proof}
First suppose $G(A)$ is a cycle.
By considering an appropriate permutation similarity, we assume that
\[A=
\left[\begin{array}{cccccc}d_1&h_1&0 &\ldots&0&h_n \\
h_1&d_2&h_2&\ddots&&0\\ 0&h_2&d_3&h_3&\ddots&\vdots \\
\vdots&\ddots &\ddots &\ddots &\ddots&0\\
0&&\ddots&h_{n-2}&d_{n-1}&h_{n-1}\\
h_n&0&\ldots&0&h_{n-1}&d_{n}
\end{array}\right].
\]
where $h_i$ and $d_i$ are positive, $i=1, \dots, n$. Each minimal CP representation of $A$
consists of $n$ summands $\vvti bi$, with each $\vi bi$ supported by one
of the $n$ edges of the graph. Without loss of generality,
\[\vvti b1=\left[
             \begin{array}{cc}
               t & h_1 \\
               h_1 & \frac{h_1^2}{t} \\
             \end{array}
           \right]\oplus 0_{n-2},\quad t>0.
\]
Then
\[A-\vvti b1=
\left[\begin{array}{cccccc}d_1-t&0&0 &\ldots&0&h_n \\
0&d_2-\frac{h_1^2}{t}&h_2&\ddots&&0\\ 0&h_2&d_3&h_3&\ddots&\vdots \\
\vdots&\ddots &\ddots &\ddots &\ddots&0\\
0&&\ddots&h_{n-2}&d_{n-1}&h_{n-1}\\
h_n&0&\ldots&0&h_{n-1}&d_{n}
\end{array}\right]\]
is a completely positive matrix which has a CP representation with $n-1$ summands,
and its graph is a path on $n$ vertices. In particular, \[\rank (A-\vvti b1\,)=\cpr (A-\vvti b1\,)=n-1.\]
By the previous result, $A-\vvti b1$ has a unique CP representation. Hence the number of
rank one representations of $A$ is equal to the number of positive $t$'s for which $A-\vvti b1$ is
doubly nonnegative and singular.
The last $n-2$ columns in the positive semidefinite $A'=A-\vvti b1$
are linearly independent
by their sign pattern, hence the principal submatrix $A'(1,2)=A(1,2)$ is nonsingular.
As   $\rank A'=n-1$, the matrix  $A'/A(1,2)$ is positive semidefinite and has rank $1$.
This latter matrix has the form
\[A'/A(1,2)=\left(
         \begin{array}{cc}
           a-t & b \\
           b & c-\frac{h_1^2}{t} \\
         \end{array}
       \right),
\]
for some real numbers $a, b, c$. Hence $\det(A'/A(1,2))=0$ translates into a quadratic equation in $t$,
that has at most two positive solutions.

The result for a unicylic graph may be deduced by induction on $n-k$, by removing at the induction
step one vector from the minimal CP representation, supported by a pendant edge, as in the previous proof.
\end{proof}

\section{Completely positive matrices with a triangle free graph}\label{sec:main}

In this section we prove our main result. We begin with some background material to put it in context.
In \cite{DrewJohnsonLoewy1994} a simple non-combinatorial sufficient condition for complete positivity was proved:
If $A$ is a symmetric nonnegative matrix whose comparison matrix $M(A)$ is positive semidefinite,
then $A$ is completely positive. In the case that the graph of $A$ is triangle free, this sufficient
condition is also necessary.  In \cite{BermanShaked1998} a simple alternative proof for this latter
theorem was suggested.
We describe it here, since this idea will be used in the next proof.
If $G(A)$ is triangle free, then the support of every column in a CP factorization of $A$
is contained in one of the edges of $G(A)$.
For a nonnegative vector $\vc b$  with $|\supp \vc b|=2$, we have $M(\vvt b\,)=\vvt d$, where $\vc d$ is
obtained from $\vc b$ by reversing the sign of exactly one nonzero element in $\vc b$.
Therefore if $A=\sum_{i=1}^k \vvti bi$, where for each $i$ the support of $\vi bi$ is of size at most $2$, we have
$M(A)=\sum_{i=1}^k \vvti di$, where $\vvti di=M(\vvti bi\,)$ is a rank one positive semidefinite matrix,
with $\supp \vi di=\supp \vi bi$, $|\vi di|=\vi bi$.
(If $|\supp \vi bi|=1$ simply take $\vi di=\vi bi$.) Hence $M(A)$ is positive semidefinite.

Our main result concerns exactly these completely positive matrices with triangle free graph.

\begin{theorem} \label{thm:trianglefree}
Let $A$ be an $n\times n$ irreducible completely positive matrix, $n\ge 2$, whose graph is triangle free.
Then the following are equivalent
\begin{enumerate}
\item[{\rm(a)}] $M(A)$ is singular.
\item[{\rm(b)}] $A$ has a unique CP factorization.
\item[{\rm(c)}] $A$ has a unique minimal CP factorization.
\end{enumerate}
\end{theorem}

\begin{remark}\label{rem:bipartite}
{\rm If the triangle free  $G(A)$ is bipartite, then $A$ is singular if and only if $M(A)$ is singular:
Let $\alpha_1$ and $\alpha_2$ be the independent sets of vertices in $G(A)$ such that
every edge has one end in $\alpha_1$ and the other in $\alpha_2$. Then
$A=SM(A)S$, where $S$ is the signature matrix with
\[s_{ii}=\left\{\begin{array}{rl}1,\quad &i\in \alpha_1\\
-1,\quad&i\in \alpha_2   \end{array}\right. .\]
In particular, in the case that $G(A)$ is a tree, (a)~$\Rightarrow$~(c) of Theorem \ref{thm:trianglefree}
is the first part of Theorem \ref{th:trees}.

Also, if $G(A)$ is bipartite and  $A$ is a nonsingular, then for every $i$
there exists $\delta>0$ such that $A-\delta E_{ii}$ is a singular completely positive matrix. To see that,
assume for convenience that $i=n$. Then
\[\det (A-\delta E_{nn})=\det A-\delta \det A(n).\]
As $A(n)$ is nonsingular we get that for $\delta =\det A/\det A(n)$ the matrix $A'=A-\delta E_{nn}$ is singular. It is
positive semidefinite, since it has a nested sequence of $n$ leading principal minors, all positive except for the
determinant of the matrix itself, which is zero. Hence also the $n$-th diagonal element of $A'$ is nonnegative,
and $A'$ is a doubly nonnegative matrix, and its graph is   $G(A)$.
}
\end{remark}

Both  observations in Remark \ref{rem:bipartite} will be used in the proof of Theorem \ref{thm:trianglefree}.

\begin{proof}[Proof of Theorem \ref{thm:trianglefree}]
Only (a)~$\Rightarrow$~(b) and (c)~$\Rightarrow$~(a) need to be proved.

(a)~$\Rightarrow$~(b): Let
$A=\sum_{i=1}^k \vvti bi$ be any CP representation of $A$. The support of each $\vi bi$ is contained in an edge of $G(A)$, hence $1\le|\supp \vi bi|\le 2$ for every $i$.
Then, as explained above,
\[M(A)=\sum_{i=1}^k M\left(\vvti bi\,\right),\]
where $ M\left(\vvti bi\right)=\vvti di$ is a rank one positive semidefinite matrix,
with $\supp \vi di=\supp \vi bi$, $|\vi di|=\vi bi$.

The irreducible $M$-matrix $M(A)$ is singular, hence it has a positive
eigenvector $\vc v$ corresponding to the eigenvalue zero.
Then from $M(A)=\sum_{i=1}^k\vvti di$,  we get that each $\vi di $ is in the column space of $M(A)$, and is therefore orthogonal to
$\vc v$. Since $1\le |\supp\vi di|\le 2$, this forces $\sigma_i=\supp\vi di$ to have size $2$, and hence
$\vi di[\sigma_i]$ spans the orthogonal complement  $\{\vc v[\sigma_i]\}^\bot$ in $\bR^2$. As the  vectors $\vdn b1k$ are pairwise
linearly independent, so are $\vdn d1k$. Combined with the fact that
$\{\vc v[\sigma_i]\}^\bot$  is one dimensional,
this implies that $\supp \vi dj\ne \sigma_i$ for every $j\ne i$.
Therefore if $\sigma_i=\{p,q\}$, we have
$\left(\vvti di \,\right)_{pq}=(M(A))_{pq}$. This equality, together with $\vit di\vc v=0$,
completely determines $\vi di$  up to sign, and hence also determines  $\vi bi=|\vi di|$. This shows that the representation
 $A=\sum_{i=1}^k \vvti bi$ is unique (and necessarily $k=|E(G(A))|$).

(c)~$\Rightarrow$~(a): We will show that if $M(A)$ is nonsingular, then
$A$ has at least two minimal CP factorizations.
As the number of CP factorizations is   preserved by positive diagonal congruence, it suffices to prove
the following claim:
If $M(A)$ is diagonally dominant and nonsingular,
then $A$ has at least two different minimal CP representations. We prove the claim by
induction on the number $m$ of edges of the connected  $G(A)$. Recall that
the number of summands in a minimal CP representation of such $A$ is $m$.

Consider first the case that $G(A)$ is a tree, $m=n-1$. As $M(A)$ is a nonsingular diagonally dominant matrix,   $A$ is also nonsingular
by Remark \ref{rem:bipartite}. Since $G(A)$ is a tree, $\cpr A=\rank A$,
and any minimal CP representation of $A$ has $n$ summands.
We can generate two different minimal CP representations as follows.
Choose $\delta_1>0$ such that $A-\delta_1 E_{11}$ is positive semidefinite and singular, and
$\delta_2>0$ such that $A-\delta_2 E_{22}$ is positive semidefinite and singular ($\delta_i=\det(A)/\det (A(i))$, $i=1,2$). Both matrices
are doubly nonnegative, with the same tree graph as $A$, and are therefore completely positive.
By  Remark \ref{rem:bipartite}, both their comparison matrices are also singular.
As (a) implies (b), there exist unique CP representations
\[A-\delta_1 E_{11}=\sum_{i=1}^{n-1} \vvti bi
~~\text{ and }~~
 A-\delta_2 E_{22}=\sum_{i=1}^{n-1}\vvti ci,\]
where for each $i$ the nonnegative vectors $\vi bi $ and $\vi ci$ are supported by the same edge of the tree.
Then
\[A=\delta_1 E_{11}+\sum_{i=1}^{n-1} \vvti bi ~~\text{ and }~~  A=\delta_2 E_{22}+\sum_{i=1}^{n-1} \vvti ci \]
are two  minimal rank $1$ representations of $A$ that are clearly different (by the only vector in each representation
that is supported by a single vertex).

Now suppose the claim holds for $m\ge n-1$, and let $A$ be an
irreducible completely positive matrix with a nonsingular diagonally dominant comparison matrix, whose triangle free graph $G(A)$ has $m+1$ edges.
As $M(A)$ is nonsingular, there exists $1\le i\le n$ such that in row $i$ there is strict diagonal dominance, i.e.,
$(M(A)\vc e)_i>0$.
As $G(A)$ is not a tree,
there exists an edge whose removal from $G(A)$ will keep the graph connected.
Without loss of generality  $\{1,2\}$ is such an edge. Let $F=\left[
                                                          \begin{array}{rr}
                                                            a_{12} & -a_{12} \\
                                                            -a_{12} & a_{12} \\
                                                          \end{array}
                                                        \right]\oplus 0_{n-2}$.
Since $F\vc e=\vc 0$,  we have $(M(A)-F)\vc e=M(A)\vc e\ge \vc 0$ and
\[((M(A)-F)\vc e)_i=(M(A)\vc e)_i>0.\]
Hence $M(A)-F=M(A-|F|)$ is a nonsingular diagonally dominant matrix whose graph has $m$ edges. By the induction hypothesis, $A-|F|$ has at least
two different minimal CP representations,
\[A-|F|=\sum_{i=1}^{m-1} \vvti bi ~~\text{ and }~~
A-|F|=\sum_{i=1}^{m-1} \vvti ci.\]
Then
\[A=|F|+\sum_{i=1}^{m-1} \vvti bi ~~\text{ and }~~  A= |F|+\sum_{i=1}^{m-1} \vvti ci \]
are two different minimal CP representations of $A$.
\end{proof}

In the case that $G(A)$ is a tree,  $A$ has infinitely many
minimal CP factorizations by Theorem \ref{th:trees}. The following corollary extends this to a bipartite $G(A)$.

\begin{corollary}\label{cor:bipartite}
Let $A$ be an $n\times n$ irreducible completely positive matrix, whose graph is bipartite. Then
\begin{enumerate}
\item[{\rm(a)}] If $A$ is singular, it has a unique CP factorization.
\item[{\rm(b)}] If $A$ is nonsingular, it has at least two
minimal CP factorizations, and infinitely many CP factorizations.
\end{enumerate}
 \end{corollary}

\begin{proof}
By Remark \ref{rem:bipartite}, in this case $M(A)$ is singular if and only if $A$ is.
Therefore (a), and the first part of (b), follow  from Theorem \ref{thm:trianglefree}.

If $A$ is nonsingular, we may choose $\delta>0$
such that $A'=A-\delta E_{11}$ is singular and positive semidefinite. Then $A'$ is doubly nonnegative
with $G(A')=G(A)$ bipartite, and therefore $A'$ is completely positive and singular.
By part (a), $A'$ has a unique CP representation
\[A'=\sum_{i=1}^m \vvti bi,\]
where each $\vi bi$ is supported by an edge in $G(A')$.
Without loss of generality, $1\in \supp \vi b1$. Then
\[A=\delta E_{11}+\vvti b1+\sum_{i=2}^k\vvti bi.\]
As $\delta E_{11}+\vvti b1=(\sqrt{\delta}\vi e1)(\sqrt{\delta}\vi e1)^T+\vvti b1$, and $\supp \sqrt{\delta}\vi e1\subseteq \supp \vi b1$,
the result follows from Observation \ref{ob:b1b2}.
\end{proof}

In the case that $G(A)$ is an $n$-cycle, $n\ge 4$, by Theorem \ref{thm:cycle}
the completely positive matrix $A$ has either one or two minimal CP factorizations.
Theorem \ref{thm:trianglefree} gives us an easy way to determine how many minimal
CP factorizations $A$ has. In the next corollary we state the result
in terms of a given CP factorization of the cyclic $A$.

\begin{corollary}\label{cor:cycle}
Let $A$ be a completely positive matrix whose graph is the $n$-cycle $1-2-\cdots-n-1$.
Then the CP factorization $A=BB^T$, where
\begin{equation}B=\left(
      \begin{array}{ccccc}
        s_1 & 0 & 0 & \dots & t_n \\
        t_1 & s_2 & 0 &   & \vdots \\
        0 & t_2 & \ddots &  &   \\
        \vdots & 0 & \ddots & s_{n-1} & 0 \\
        0 & 0 & 0 & t_{n-1} & s_n \\
      \end{array}
    \right), \quad s_i,t_i>0,~i=1, \dots, 5, \label{eq:Bcyclic}
\end{equation}
is the unique CP factorization of $A$ if  and only if
$\Pi_{i=1}^n s_i=\Pi_{i=1}^n t_i$.
\end{corollary}

\begin{proof}
We have $M(A)=SS^T$, where $S$ is obtained from $B$ by replacing each $t_i$ by $-t_i$.
Therefore $\det M(A)=\left(\Pi_{i=1}^n s_i-\Pi_{i=1}^n t_i\right)^2$ (whether $n$ is even or odd!),
and $M(A)$ is singular if and only if $\Pi_{i=1}^n s_i=\Pi_{i=1}^n t_i$.
\end{proof}

Although Theorem \ref{thm:trianglefree} holds for matrices whose
graph is triangle free,  in some cases it has implications on the number of
(minimal) CP factorizations of other completely positive matrices that lie on
the boundary of the cone $\cp_n$. Getting into that requires a brief
reminder regarding  $\cp_n$, its dual $\cop_n$
and the matrices on their boundaries.
The inner product in $\sym_n$ is $\langle A, B\rangle=\trace(AB)$. The mutual duality of the cones $\cp_n$
and $\cop_n$ is with respect to this inner product:
\[\cp_n^*=\left\{X\in \sym_n\mid \langle X, Y\rangle \ge 0 \text{ for every } Y\in \cp_n\right\}=\cop_n,\]
and
\[\cop_n^*=\left\{X\in \sym_n\mid \langle X, Y\rangle \ge 0 \text{ for every } Y\in \cop_n\right\}=\cp_n.\]
In particular, matrices on the boundary of $\cp_n$ are orthogonal to matrices on the boundary of $\cop_n$.
So if $M$ is on the boundary of $\cop_n$, there exists a nonzero nonnegative vector $\vc b$ such that
\[ \vct bM\vc b=\langle M, \vvt b\,\rangle=0.\]
Such $\vc b$ is called a {\it zero} of $M$, and if $\supp \vc b$ does not strictly contain a support
of another zero, $\vc b$ is called a {\it minimal zero} of $M$. In \cite{Hildebrand2014} it
was proved that each copositive $M$ has a finite number of minimal zeros, up to multiplication by a scalar, and
every zero of $M$ is a nonnegative combination of a finite number of minimal zeros. We call a set of
minimal zeros of $M$ {\it representative} if every minimal zero of $M$ is a scalar multiple of exactly
one vector in the set. Symmetric nonnegative matrices are copositive, and so are the
positive semidefinite matrices. For $n\ge 5$ there exist  copositive matrices in $\cop_n$  that are not a sum of a positive semidefinite
matrix and a symmetric nonnegative one. Such a copositive matrix  is called {\it exceptional}.
From \cite{Hildebrand2014} we know that if an exceptional  matrix in $\cop_n$ generates an extreme ray of the cone (i.e., it is {\it extremal}),
then it has at least $n$ representative minimal zeros that span $\bR^n$.

\begin{lemma}\label{lem:WX}
Let $M$ be an $n\times n$ exceptional extremal copositive matrix, that has exactly $n$ representative minimal zeros,
 $\vdn w1n$. Let
$W=\left[\begin{array}{cccc}\vi w1&\vi w2&\dots&\vi wn \end{array}\right]$.
Then each completely positive matrix $A$, which is orthogonal to $M$ in $\sym_n$,
has the form $A=WCW^T$ for some $n\times n$ completely positive $C$, $\cpr A=\cpr C$, and
and the number of (minimal) CP factorizations of $A$ is equal to the number of (minimal)
CP factorizations of $C$.
\end{lemma}

\begin{proof}
By the assumptions on $M$, the nonnegative matrix $W$ is nonsingular.
Each zero of $M$ is a nonnegative combination
of $W$'s columns. If $A=BB^T$ for some nonnegative $B$, each column of $B$ is a zero of $M$,
hence $B=WX$, where $X\ge 0$,
and \begin{equation}A=(WX)(WX)^T=WCW^T,\label{eq:AC}\end{equation} where $C=XX^T$ is a completely positive matrix.
By \eqref{eq:AC}, $\cpr A\le \cpr C$.

If $A=QQ^T$ is any CP factorization of $A$, then
$Q=WY$, $Y\ge 0$.
Since \[A=QQ^T=W(YY^T)W^T\] and $W$ is nonsingular,  we get that
$YY^T=W^{-1}AW^{-T}=C$ is a CP factorization of $C$.  In particular, if $A=QQ^T$ is a minimal
CP factorization, this shows that $\cpr C\le \cpr A$, and we conclude that $\cpr A=\cpr C$.
As  $W$ is nonsingular,  $B\ne Q$ if and only if $X\ne Y$.
Hence the number of CP factorizations/minimal CP factorizations of both matrices are equal.
\end{proof}

Given a completely positive matrix orthogonal to an exceptional
extremal copositive matrix, which has $n$ representative minimal zeros, we may
compute $C=W^{-1}AW^{-T}$. If the matrix $C$ has a triangle free graph, it is
easy to check by its comparison matrix whether it has a unique CP factorization.
However, for general $n$ finding the exceptional extremal matrices in
$\cop_n$ is a major open problem. These were characterized fully
only for $n=5$ in \cite{Hildebrand2012} and for $n=6$ in \cite{AfoninDickinsonHildebrand2020}.
But some examples of exceptional extremal matrices
are known for every $n$.

In $\cop_5$ one exceptional extremal matrix is the Horn matrix
\[H=\left[\begin{array}{rrrrr}1&-1&1&1&-1\\ -1&1&-1&1&1\\
1&-1&1&-1&1\\1&1&-1&1&-1\\-1&1&1&-1&1\end{array}\right].\]
Other exceptional extremal matrices in $\cop_5$ are the matrices called Hildebrand matrices. We do
not describe these here, but  mention that each Hildebrand matrix has exactly five zeros,
up to scalar multiplication; all these zeros are minimal, and each has support of size $3$. Every   exceptional extremal matrix in $\cop_5$
is obtained from either the Horn matrix or from a Hildebrand matrix by permutation similarity
and diagonal scaling.

We can demonstrate the use of the Lemma \ref{lem:WX} in conjunction with Theorem \ref{thm:trianglefree}
for some $5\times 5$ matrices.

\begin{example}\label{ex:Hornorth}
{\rm The minimal zeros of the Horn matrix $H$ are
\[\vi wi=\vi ei+\vi e{i\plmod 1}, \quad i=1, \dots, 5,\]
where $\plmod$ denotes addition modulo 5 on $\{1,\dots, 5\}$.
Let
\[W=\left[\begin{array}{ccccc}\vi w1&\vi w2&\vi w3&\vi w4&\vi w5 \end{array}\right]=\left[\begin{array}{ccccc}1&0&0&0&1\\ 1&1&0&0&0\\
0&1&1&0&0\\0&0&1&1&0\\0&0&0&1&1\end{array}\right].\]
The matrix
\[A=\left[\begin{array}{ccccc}8&5&1&1&5\\ 5&8&5&1&1\\
1&5&8&5&1\\1&1&5&8&5\\5&1&1&5&8\end{array}\right]\]
is an example of a completely positive matrix orthogonal
to $H$ (introduced in \cite{DuerStill2008}).
The matrix
\[C=W^{-1}AW^{-T}=\left[\begin{array}{ccccc}3&1&0&0&1\\ 1&3&1&0&0\\
0&1&3&1&0\\0&0&1&3&1\\1&0&0&1&3\end{array}\right]\]
has a cyclic graph and a strictly diagonally dominant comparison matrix, hence $C$,
and therefore $A$,   has two minimal CP factorizations. Moreover, $C$, and therefore $A$,
has infinitely many CP factorizations. For example,
for every $0<t<1$ the matrix $C-tI$ has two minimal CP factorizations.
Let $C-tI=U(t)U(t)^T$ be a minimal CP factorization of $C-tI$ (one of the two).
Then \[C=\left[\begin{array}{cc}\sqrt{t}I_5& U(t) \end{array}\right]\left[\begin{array}{cc}\sqrt{t}I_5& U(t) \end{array}\right]^T, \quad 0<t<1\]
are infinitely many different CP factorizations of $C$, and
$A=V(t)V(t)^T$, where \[V(t)=W\left[\begin{array}{cc}\sqrt{t}I_5& U(t) \end{array}\right],\]
are  infinitely many CP factorizations of $A$ (not all).

On the other hand,
\[A'=\left[\begin{array}{ccccc}6&4&1&1&4\\ 4&6&4&1&1\\
1&4&6&4&1\\1&1&4&6&4\\4&1&1&4&6\end{array}\right]\]
has a unique CP representation:
\[C'=W^{-1}A'W^{-T}=\left[\begin{array}{rrrrr}2&1&0&0&1\\ 1&2&1&0&0\\
0&1&2&1&0\\0&0&1&2&1\\1&0&0&1&2\end{array}\right]\]
has a singular comparison matrix. The unique CP factorization of $C'$ happens to be
\[C'=WW^T,\]
and the unique CP factorization of $A'$ is therefore
$A'=W^2(W^2)^T$, where
\[W^2=\left[\begin{array}{ccccc}1&0&0&1&2\\ 2&1&0&0&1\\
1&2&1&0&0\\0&1&2&1&0\\0&0&1&2&1\end{array}\right].\]
We mention that it was already shown in
 \cite[Theorems 4.1 \& 4.4]{ShakedBomzeJarreSchachinger2013}  that
any positive completely positive matrix $A$ orthogonal to the Horn matrix  has $\cpr A=\rank A$, and at
most two minimal CP factorization, but without the simple rule to determine the exact number. If $A$  is positive, then $A=(WB)(WB)^T$, where $B$
is as in \eqref{eq:Bcyclic}. In this case $\cpr A=\rank A=\rank B=5$. Lemma \ref{lem:WX}
combined with  Corollary \ref{cor:cycle} yield a characterization, previously obtained
in \cite{Zhang2019}, of when is a given
minimal CP factorization of such $A$ unique.

Now consider matrices in $\cp_5$ there  are orthogonal to a Hildebrand matrix. As
mentioned above a Hildebrand matrix has
exactly five minimal zeros, all minimal, up to scalar multiplication,
any completely positive matrix orthogonal to a Hildebrand matrix has
a unique factorization $A=VDV^T$, where $V$ is the matrix of representative
zeros of the Hildebrand matrix, and $D$ is a nonnegative diagonal matrix. As
such $D$ has a unique CP factorization (as a direct sum of rank $1$ matrices...), so does $A$. This was previously
proved (for a positive $A$) in \cite[Theorem 4.4]{ShakedBomzeJarreSchachinger2013}.}
\end{example}

These examples  are not unique to $\cp_5$.
Results from \cite{Hildebrand2017} show that
exceptional extremal ``Hildebrand-like" matrices exist in $\cop_n$ for every
odd order $n\ge 5$ (with $n$  representative zeros, all minimal, whose supports
are the sets $\{1, \dots, n\}\setminus \{i,i\plmod 1\}$, where $\plmod$
denotes addition modulo $n$ on $\{1, \dots n\}$). Completely positive
matrices orthogonal to these have a unique CP factorization.
By \cite{Hildebrand2017},
exceptional extremal ``Horn-like" matrices exist in $\cop_n$ for every $n\ge 5$
(where each zero is a sum of two minimal zeros, and there are zeros
whose supports are the sets $\{1, \dots, n\}\setminus \{i,i\plmod 1\}$, where $\plmod$
denotes addition modulo $n$ on $\{1, \dots n\}$). Depending on the
minimal zeros, the example of completely positive matrices orthogonal to
the Horn matrix with either one or two minimal CP factorizations
may be recreated in these orders. Such examples definitely exist in $\cop_6$
by \cite[Theorem 8.1]{Hildebrand2017}.

\section{The minimal face of $\cp_n$ containing $A$}\label{sec:faces}
The facial structure of the completely positive cone is of interest, and not yet
thoroughly explored. We recall that a subcone $\F$ of a convex cone $\K$ is
a  {\it face} of $\K$ if $X, Y\in \K$ and $X+Y\in \F$  implies that both $X$
and $Y$ are in $\F$.
In this section we demonstrate the implication of the existence of a unique CP factorization
regarding the minimal face of $\cp_n$ containing a given completely positive matrix $A$.
Let us denote the minimal face of $\cp_n$ containing  $A$ by $\F^A_{\cp_n}$. It is
 a face of $\cp_n$ that contains $A$
and is contained in every face  that contains $A$ (it is the intersection of all such faces).
For an arbitrary completely positive matrix $A$ on the boundary of $\cp_n$, the description
of $\F^A_{\cp_n}$ is as follows.

\begin{theorem}\label{thm:F^A}
Let $A\in \cp_n$. Then the minimal
face of $\cp_n$ containing $A$ is
\[\F^A_{\cp_n}=\cone\left\{\vvt b\in \R^n_+ \mid A-\vvt b\in \cp_n \right\}.\]
\end{theorem}

\begin{proof}
Let us denote $\K=\cone\left\{\vvt b\in \bR^n_+ \mid A-\vvt b\in \cp_n \right\}$.
 To show
that $\K$ is the minimal face containing $A$, we will show that it is a face, and any face containing
$A$ contains the whole  $\K$.
We start with the latter.
Let $\F$ be a face of $\cp_n$ containing $A$, and let $\vc b$ be any nonnegative vector such that
$A-\vvt b$ is completely positive, then  \[A=\vvt b+ ( A-\vvt b \,)\in \F\] implies
that $\vvt b\in \F$. Hence $\left\{\vvt b\in \R^n_+ \mid A-\vvt b\in \cp_n \right\}\subseteq \F$, and since $\F$ is
a cone, $\K\subseteq \F$.

Now we show that  $\K$ is a face of $\cp_n$.
First, if $A=B+C$, $B, C\in \cp_n$, and
$B=\sum_{i=1}^k\vvti bi$ is a CP representation of $B$, then for each
$1\le i\le k$ the matrix  \[A-\vvti bi=\sum_{\substack{1\le j\le k\\j\ne i}}\vvti bj+C\] is completely
positive. Hence each $\vi bi\in \K$, and therefore $B=\sum_{i=1}^k\vvti bi \in \K$. Similarly $C\in \K$.

Now suppose $X$ is any matrix in $\K$. Then $X=\sum_{i=1}^m t_i\vvti vi$, where $\vi vi \in \bR^n_+$, $t_i>0$ and $A-\vvti vi\in \cp_n$
for every $i$. Let $t=\sum_{i=1}^kt_i$.
If $X=Y+Z$, where $Y,Z\in \cp_n$, write
\[A=\frac{1}{t}Y+\frac{1}{t}\left(Z+\sum_{i=1}^m t_i (A- \vvti vi) \right).\]
Then $\frac{1}{t}Y\in \K$, and reversing the roles of $Y$ and $Z$, also $\frac{1}{t}Z\in \K$.
Hence $\K$ is a face.
\end{proof}

Note that checking whether $\vvt b$ is in this face
requires checking complete positivity of a matrix, hence it is not an easy task. In the common
case that $A$ has infinitely many CP factorizations, $\F^A_{\cp_n}$ has infinitely many extreme rays.
In the special case that $A$ has a unique CP representation, the description of $\F^A_{\cp_n}$ is
more tangible, and $\F^A_{\cp_n}$ is a polyhedral cone.

\begin{corollary}\label{cor:FAunique}
Let $A\in \cp_n$ have a unique CP representation $A=\sum_{i=1}^k \vvti bi$. Then the minimal
face of $\cp_n$ containing $A$ is
\[\F^A_{\cp_n}=\cone\left\{\vvti bi\mid 1\le i\le k \right\}.\]
\end{corollary}

\begin{proof}
For a nonnegative vector $\vc b\in \R^n$, $\vvt b$ may be completed to a CP representation
of $A$ if and only if $\vc b$ is one of $\vdn b1k$.
\end{proof}

\def\cprime{$'$}

\end{document}